\documentclass[11pt]{article}
\usepackage[utf8]{inputenc}
\usepackage[margin=1.4in]{geometry}

\usepackage{amsmath, amsthm, amssymb, bbm, mathrsfs, url}
\usepackage{algpseudocode}

\newtheorem{thm}{Theorem}
\newtheorem{lem}{Lemma}
\newtheorem{ex}{Example}

\theoremstyle{definition}
\newtheorem{definition}{Definition}

\newcommand{\thistheoremname}{}
\newtheorem*{genericthm}{\thistheoremname}

\theoremstyle{remark}
\newtheorem*{remark}{Remark}

\DeclareMathOperator{\disc}{disc}
\DeclareMathOperator{\db}{db}
\newcommand{\M}{\textsc{M}}
\newcommand{\oct}{\textsc{Oct}}
\newcommand{\F}{\mathbb{F}}
\newcommand{\Z}{\mathbb{Z}}
\newcommand{\R}{\mathbb{R}}
\DeclareMathOperator{\Exp}{\mathbb{E}}
\newcommand{\Dcal}{\mathcal{D}}
\newcommand{\Dfrak}{\mathfrak{D}}
\newcommand{\xbf}{\mathbf{x}}
\newcommand{\ybf}{\mathbf{y}}
\newcommand{\Pcal}{\mathcal{P}}
\DeclareMathOperator{\supp}{supp}
\renewcommand{\sf}{\mathfrak{s_f}}
\newcommand{\cubem}{\{0,1\}^{d+1} \!\setminus\! \{0\}}
\DeclareMathOperator{\dev}{dev}

\newcommand{\T}{\mathrm{T}}
\newcommand{\Ical}{\mathcal{J}}

\begin{document}

\title{Quasirandom additive sets and Cayley hypergraphs}
\date{{\small CWI \& QuSoft}\\[2ex] \today}
\author{Davi Castro-Silva}

\maketitle

\begin{abstract}
    We study the interplay between notions of quasirandomness for additive sets and for hypergraphs.
    In particular, we show a strong connection between the notions of Gowers uniformity in the additive setting and discrepancy-type measures of quasirandomness in the hypergraph setting.
    Exploiting this connection, we provide a long list of disparate quasirandom properties regarding both additive sets and Cayley-type hypergraphs constructed from such sets, and show that these properties are all equivalent (in the sense of Chung, Graham and Wilson) with polynomial bounds on their interdependences.
\end{abstract}

\section{Introduction}


Quasirandom properties can be informally thought of as certificates of randomness for the object in consideration.
Given some class of combinatorial objects, such as graphs or 3-uniform hypergraphs, we say that a (deterministic) property of these objects is a \emph{quasirandom property} if it satisfies two conditions:
a uniformly random object from that class satisfies this property with high probability;
and any object which satisfies this property will also behave in many other ways like a random object.
In such cases, just by knowing that an object satisfies some quasirandom property, one gleans a wealth of information about its behaviour in many respects;
such objects are then said to be \emph{quasirandom}.

The notion of quasirandomness was originally introduced in the setting of graphs, in a seminal paper of Chung, Graham and Wilson~\cite{QuasirandomGraphs}.
These authors considered several natural properties typically satisfied by random graphs -- such as having the expected subgraph counts, having uniform edge-distribution over vertex cuts or having large spectral gap -- and showed that all of them are (in a specific sense) equivalent to each other.
Following their work, the study of quasirandom properties has been extended to several other combinatorial classes of objects;
we refer the reader to Chung's website~\cite{ChungSite} for a long list of references.
In this paper we will consider hypergraphs and additive sets (i.e. subsets of additive groups), focusing on the relationships between their respective quasirandom properties.

\subsubsection*{Quasirandom hypergraphs}

Following the introduction of quasirandom graphs, Chung and Graham~\cite{QuasirandomHypergraphs, QuasiSetSystems} and Kohayakawa, R\"odl and Skokan~\cite{HypergraphQuasirandomnessRegularity} undertook the task of extending such notions to hypergraphs.
They considered $k$-uniform hypergraphs ($k$-graphs) which mimic the random hypergraph $G^{(k)}(n, p)$, where each set of~$k$ elements in $[n] := \{1, 2, \dots, n\}$ is chosen to be an edge independently with probability~$p$.
This lack of correlation in the presence of edges leads to strong uniformity properties which make random hypergraphs easy to handle;
some of these good properties were then shown to form a rough equivalence class of random-like characteristics.


The central notion in the work of Chung and Graham was the \emph{deviation} of a hypergraph.
Given a $k$-graph $H$, we write $v(H)$ and $e(H)$ to denote the number of vertices and edges in~$H$ (respectively), and write $\delta(H)$ to denote its edge density.
The deviation of a hypergraph can be seen as a measure of how much its edge distribution deviates from the random distribution it is supposed to mimic;
it is formally defined by
$$\dev_k(H) = \mathbb{E}_{\xbf^{(0)}, \xbf^{(1)} \in V(H)^k} \prod_{\omega \in \{0, 1\}^{k}} \big( H(x_1^{(\omega_1)}, \dots, x_{k}^{(\omega_{k})}) - \delta(H) \big),$$
where $H(x_1, \dots, x_k)$ denotes the edge indicator function $\mathbbm{1}\big[\{x_1, \dots, x_k\} \in E(H)\big]$.
It is always true that $0\leq \dev_k(H)\leq 1$, and one can show that random hypergraphs will have very small deviation with high probability.

By contrast, the central concept in the work of Kohayakawa, R\"odl and Skokan was the \emph{discrepancy} of a hypergraph, which quantifies how far from uniformly distributed its edges are when measured against lower-order structures.
In their paper these lower-order structures were given by the $k$-cliques of $(k-1)$-graphs, but here we will work with the slightly more general notion of $(k-1)$-cuts.
The discrepancy of a $k$-graph~$H$ is then defined by
$$\disc_{k-1}(H) = \max_{S_1, \dots, S_k \subseteq V(H)^{k-1}} \bigg| \Exp_{\xbf \in V(H)^k} \bigg[\big(H(\xbf) - \delta(H)\big) \prod_{i=1}^k S_i\big((x_j)_{j\neq i}\big) \bigg]\bigg|,$$
where we use the same notation for a set~$S$ and for its indicator function $\mathbbm{1}[x\in S]$.
It is not hard to show that random hypergraphs will have very small discrepancy with high probability.
Both deviation and discrepancy can be seen as measures of quasirandomness.

Another statistic which can be accurately estimated in random hypergraphs is the count of various smaller hypergraphs occurring as a subgraph.
Given two hypergraphs~$F$ and~$H$, denote the number of labelled copies of~$F$ in~$H$ by $N_F(H)$.
If~$H$ is the random hypergraph $G^{(k)}(n, p)$,
then the expected value of $N_F(H)$ is
$$p^{e(F)} n (n-1) \dots (n-v(F)+1) = p^{e(F)} n^{v(F)} + O(v(F)^2 n^{v(F)-1});$$
moreover, $N_F(H)$ is highly concentrated around this expected value.
It was shown by Chung and Graham~\cite{QuasirandomHypergraphs, QuasiSetSystems}, and by Kohayakawa, R\"odl and Skokan~\cite{RegularityHypergraphsQuasirandomness}, that large hypergraphs~$H$ which have either small deviation ($\dev_k(H) = o(1)$) or small discrepancy ($\disc_{k-1}(H) = o(1)$) must contain approximately the expected count of \emph{all} subgraphs of bounded size:
$N_F(H) = \delta(H)^{e(F)} v(H)^{v(F)} + o(v(F)^2 v(H)^{v(F)})$.

A special role in their results played by the \emph{octahedra}.
The $k$-octahedron $\oct^{(k)}$ is the complete $k$-partite $k$-graph where each vertex class has precisely two vertices.
Note that the deviation of a $k$-graph~$H$ can be interpreted as an average weighted count of octahedra $\oct^{(k)}$, when the weight is given by the balanced indicator function $H(x_1, \dots, x_k) - \delta(H)$.
It was shown in~\cite{HypergraphQuasirandomnessRegularity} that $\oct^{(k)}$ is \emph{complete} for the notions of quasirandomness given above:
any $k$-graph~$H$ which has approximately the `correct' proportion of subgraphs isomorphic to $\oct^{(k)}$ -- meaning the expected proportion in a random hypergraph of the same edge density -- will be quasirandom in their sense;
in particular, it follows that~$H$ will also have approximately the correct proportion of \emph{every other} fixed $k$-graph~$F$ (of bounded size) as a subgraph.

\smallskip

Despite these results, and in contrast to the simpler setting of graphs, it turns out that there are several distinct equivalence classes of quasirandomness notions for hypergraphs.
These different classes and their inter-relationships were studied by Chung~\cite{QuasiClasses}, Kohayakawa, Nagle, R\"odl and Schacht~\cite{WeakHypergraphRegularity}, Conlon, H\`an, Person and Schacht~\cite{WeakQuasirandomness}, Lenz and Mubayi~\cite{PosetQuasirandomness} and Towsner~\cite{SigmaAlgebrasHypergraphs}, to name a few.

Let $d$ and $k$ be integers with $1\leq d < k$, and let~$H$ be a $k$-uniform hypergraph.
The \emph{$d$-discrepancy} of~$H$ is defined by
$$\disc_{d}(H) = \max_{S_B \subseteq V(H)^{d}:\, B\in \binom{[k]}{d}} \Bigg| \Exp_{\xbf \in V(H)^k} \Bigg[\big(H(\xbf) - \delta(H)\big) \prod_{B\in \binom{[k]}{d}} S_B\big((x_j)_{j\in B}\big) \Bigg]\Bigg|,$$
where the maximum is taken over all collections of $\binom{k}{d}$ subsets of $V(H)^{d}$ indexed by the $d$-subsets of~$[k]$.
This is a measure of how far from uniformly distributed the edges of~$H$ are when measured against structures of order~$d$;
if the $d$-discrepancy of~$H$ is small, we think of it as being quasirandom of order~$d$.
More formally, we say that~$H$ is \emph{$\varepsilon$-quasirandom of order~$d$} if $\disc_{d}(H) \leq \varepsilon$.

The notion of deviation can also be generalized to other orders.
We define the \emph{$d$-deviation} of~$H$, denoted $\dev_{d}(H)$, by
$$\mathbb{E}_{\xbf^{(0)}, \xbf^{(1)} \in V(H)^{d}} \mathbb{E}_{y_{d+1}, \dots, y_k \in V(H)} \prod_{\omega \in \{0, 1\}^{d}} \big( H(x_1^{(\omega_1)}, \dots, x_{d}^{(\omega_{d})}, y_{d+1}, \dots, y_k) - \delta(H) \big).$$
Just as the $k$-deviation $\dev_k(H)$ can be seen as a weighted count of octahedra $\oct^{(k)}$, so can the $d$-deviation be seen as a weighed count of \emph{squashed octahedra} $\oct^{(k)}_d$, formed by
adding $k-d$ extra vertices to $\oct^{(d)}$ and attaching them to each of the edges.
Chung~\cite{QuasiClasses} showed that every $k$-graph which has small $d$-deviation must also have small $(d-1)$-discrepancy;
however, as observed by Lenz and Mubayi~\cite{PosetQuasirandomness}, the other direction does not hold, and these two notions are distinct.

Kohayakawa, Nagle, R\"odl and Schacht~\cite{WeakHypergraphRegularity} proved that quasirandomness of order~1 (also called weak quasirandomness) is sufficient for controlling the count of every \emph{linear hypergraph}, meaning those where every pair of edges intersect in at most one vertex.
Subsequently, Conlon, H\`an, Person and Schacht~\cite{WeakQuasirandomness} showed that there exists a linear hypergraph $\M^{(k)}_1$ which is \emph{complete} for the notion of weak quasirandomness:
if a hypergraph~$H$ contains approximately the expected number of subgraphs isomorphic to~$\M^{(k)}_1$, then it is necessarily weakly quasirandom.
Lenz and Mubayi~\cite{PosetQuasirandomness} determined the poset of implications between several notions of quasirandomness.

These results were ultimately generalized by Towsner~\cite{SigmaAlgebrasHypergraphs}, who provided many equivalence classes of notions of quasirandomness -- including all those which had been previously studied -- and obtained the interrelationships between those classes.
He proved that quasirandomness of order~$d$ is equivalent to having the appropriate count of all \emph{$d$-linear hypergraphs}, meaning those where each pair of edges intersect in at most~$d$ vertices, and also to having small deviation of a certain type.
He also constructed a specific $d$-linear $k$-graph $\M^{(k)}_d$ which is complete for quasirandomness of order~$d$;
his results are formally stated in Section~\ref{sec:hyp_quasi} below.

\subsubsection*{Uniform additive sets}

In a different direction, the study of quasirandomness was also extended to subsets of additive groups.
In this setting there is a strong underlying structure which comes from the group operation, and the notion of quasirandomness is related to how the considered set behaves relative to this additive structure.

The first to study this setting were again Chung and Graham~\cite{QuasirandomZn}.
They considered subsets of the cyclic group~$\Z/n\Z$, and identified many natural but seemingly unrelated properties of random subsets of $\Z/n\Z$ which turn out to be all equivalent to each other.
These properties included:
having small non-trivial Fourier coefficients,
intersecting translates of large sets in the expected frequency,
having the expected number of additive quadruples\footnote{An additive quadruple is a tuple $(x_1, x_2, x_3, x_4)$ satisfying $x_1+x_2 = x_3+x_4$.},
its Cayley sum graph being a quasirandom graph,
as well as several others.
While Chung and Graham's theorem was restricted to cyclic groups, their results can in fact be generalized to any finite additive group (see~\cite{LinearQuasirandomness, QuasiSurvey}).

It was Gowers who first noticed that, much as in the setting of hypergraphs, there is a natural \emph{hierarchy} of quasirandomness notions for additive sets
(and more generally for bounded functions on additive groups).
These notions are characterized by the so-called \emph{Gowers uniformity norms}, which control the statistics of several additive patterns (such as $k$-term arithmetic progressions) inside a given set.
These norms were originally defined by Gowers in his celebrated proof of Szemer\'edi's theorem~\cite{NewProofLengthFour, NewProofSzemeredi}, and have since proven useful in many problems from additive combinatorics and theoretical computer science.

Given a finite additive group~$G$ and an integer $k\geq 2$, we define the \emph{uniformity norm~$U^k$} for functions $f: G \to \R$ by
$$\| f \|_{U^k} = \Bigg(\Exp_{x, h_1, \dots, h_k \in G} \prod_{\omega \in \{0, 1\}^k} f \bigg( x + \sum_{i=1}^k \omega_i h_i \bigg)\Bigg)^{1/2^k}.$$
If~$f: G\to \{-1, 1\}$ is a uniformly random function and~$G$ is very large, then with high probability the $U^k$-norm of~$f$ will be very small:
one can show that $\Exp \|f\|_{U^k}^{2^k} = O(1/|G|)$.
More generally, if we randomly choose a subset $A\subseteq G$ by including each element $x\in G$ independently with probability~$\delta$, then with high probability we have $\|A - \delta\|_{U^k} = O(|G|^{-1/2^k})$.

Additive sets~$A\subseteq G$ which satisfy $\|A - \delta\|_{U^k} \leq \varepsilon$ for $\delta = |A|/|G|$ are then said to be \emph{$\varepsilon$-uniform of degree $k-1$}.
(This name is motivated by the fact that such sets do not correlate with any structure of degree $k-1$.)
We informally say that~$A$ is uniform of degree $k-1$ if it is $\varepsilon$-uniform of this degree for some small $\varepsilon>0$.

The~$U^2$-norm turns out to be intimately related to Fourier coefficients:
one easily shows that
$$\|f\|_{U^2}^4 = \sum_{\gamma\in \widehat{G}} |\widehat{f}(\gamma)|^4.$$
It follows that a set $A\subseteq G$ is linearly uniform (that is, uniform of degree~1) if and only if all of its non-trivial Fourier coefficients are small;
this coincides with the notion of quasirandomness considered by Chung and Graham~\cite{QuasirandomZn}, but is only the first step in a quasirandom ladder.
The Gowers uniformity norms form a hierarchy where the $U^{k+1}$-norm is stronger than the $U^k$-norm for each $k \geq 2$.
By successively considering the properties of being uniform of degree~$d$ for each $d\geq 1$, we then obtain an infinite hierarchy of increasingly stronger notions of quasirandomness for additive sets.

The significance of these notions in additive combinatorics stems from the fact that uniformity of degree $k$ is sufficient to control the count of $(k + 2)$-term arithmetic progressions, as well as several other linear configurations said to have complexity at most $k$.
Moreover, every `non-degenerate' system of linear forms can be controlled by some uniformity norm $U^k$.
In a way, a non-negligible part of the recent research in additive combinatorics -- especially in the area of \emph{higher-order Fourier analysis} -- can be seen as the search for useful quasirandom properties equivalent to having small uniformity norm of a given degree.
(This problem is usually stated in the contrapositive:
if a bounded function $f$ has non-negligible $U^k$-norm, what kind of structural information can we learn about $f$?)

\subsubsection*{Our results: relating these two notions}

Given the many existing connections between (hyper)graph theory and additive combinatorics, one is left to wonder:
how do the natural hierarchies of quasirandomness notions in hypergraphs (given by quasirandomness of each order~$d$) and in additive groups (uniformity of degree~$d$) relate to each other?

A first step towards answering this question was given by Aigner-Horev and H\`an~\cite{LinearQuasirandomness}, who showed that there is a strong connection between the notions of linear uniformity for additive sets and weak quasirandomness for hypergraphs.
Given an integer $k\geq 2$, a finite additive group~$G$ and a subset $A\subseteq G$, define the \emph{Cayley hypergraph} $\Gamma^{(k)}_A$ as the hypergraph whose vertices are the elements of~$G$ and where $k$ vertices form an edge iff their sum lies in~$A$.
Aigner-Horev and H\`an proved that, for any fixed~$k\geq 2$, the hypergraph~$\Gamma^{(k)}_A$ is weakly quasirandom iff the set~$A$ is linearly uniform;
moreover, they obtain polynomial bounds between the associated measures of uniformity and quasirandomness.

Our first main result is to generalize this connection between uniform sets and quasirandom Cayley hypergraphs to all orders $d\geq 1$ of quasirandomness.
We will show that an additive set~$A$ which is uniform of degree~$d$ generates quasirandom hypergraphs~$\Gamma^{(k)}_A$ of order~$d$ for any edge-size $k \geq d+1$;
conversely, it suffices to know that \emph{any single one} of these hypergraphs is quasirandom of order $d$ to conclude that~$A$ is uniform of this same degree.
More precisely, we prove:

\begin{thm} \label{thm:mainintro}
Let $d \geq 1$ be an integer and $\varepsilon\in (0, 1)$.
Let $G$ be a finite additive group and $A \subseteq G$ be a subset.
\begin{itemize}
    \item[$(i)$] If $A$ is $\varepsilon$-uniform of degree $d$, then for all $k \geq d + 1$ the Cayley hypergraph $\Gamma^{(k)}_A$ is $\varepsilon$-quasirandom of order $d$.
    \item[$(ii)$] Conversely, if $\Gamma^{(k)}_A$ is $\varepsilon$-quasirandom of order $d$ for some $k \geq d+1$, then $A$ is $2\varepsilon^{c_{k, d}}$-uniform of degree $d$.
    Here we can take $c_{k, d} = 2^{-(d+2)(2d+2)^k}$.
\end{itemize}
\end{thm}

We will then use this theorem, combined with known results from the literature on quasirandom hypergraphs and elementary combinatorial arguments, to show that a number of disparate quasirandom properties regarding additive sets and their associated Cayley hypergraphs are all equivalent to each other.
This is our second main result, which we call the \emph{Equivalence Theorem}.

An informal statement of the Equivalence Theorem is given below, where we make use of terms which will only be properly defined later on;
see Theorem~\ref{thm:Cayley_equiv} for the precise statement.
Roughly speaking, $\Phi^{(k)}_d$ is the arithmetic equivalent of the hypergraph $\M^{(k)}_d$ which is complete for quasirandomness of order~$d$, and $(d+1)$-simple hypergraphs are a generalization of $d$-linear hypergraphs where every edge has a `fingerprint' of size $d+1$ unique to it.

\begin{thm}[Equivalence Theorem, informal] \label{thm:equiv_informal}
Let $d \geq 1$ be an integer.
For every fixed $k \geq d+1$ the following statements are polynomially equivalent, where $G$ is a finite additive group and $A \subseteq G$ is a subset:
\begin{itemize}
    \item[$(i)$] $A$ is uniform of degree $d$.
    \item[$(ii)$] $A$ has few patterns of type $\Phi^{(k)}_d$.
    \item[$(iii)$] $\Gamma^{(k)}_A$ is quasirandom of order $d$.
    \item[$(iv)$] $\Gamma^{(k)}_A$ correctly counts all $d$-linear hypergraphs.
    \item[$(v)$] $\Gamma^{(k)}_A$ correctly counts all $(d+1)$-simple hypergraphs.
    \item[$(vi)$] $\Gamma^{(k)}_A$ has few squashed octahedra $\oct^{(k)}_{d+1}$.
    \item[$(vii)$] $\Gamma^{(k)}_A$ has small $(d+1)$-deviation.
\end{itemize}
\end{thm}

Finally, we note an intriguing peculiarity of this last theorem:
the hypergraph properties $(iii)-(vii)$ stated are \emph{not} equivalent for general $k$-graphs, only for those which come from additive sets.
Indeed, one can show that properties $(iii)-(iv)$ are equivalent for all $k$-graphs, as are properties $(v)-(vii)$, but properties $(iii)-(iv)$ are strictly weaker than $(v)-(vii)$
(a fact which was first noticed by Lenz and Mubayi~\cite{PosetQuasirandomness}).
This gives an interesting difference between the notions of quasirandomness in hypergraphs and in additive groups,
which is present despite their strong connection given by Theorem~\ref{thm:mainintro}.

\section{Preliminaries: notions of quasirandomness}
\label{sec:prelim}

Here we collect the necessary notation which will be used throughout the paper, and formally introduce the appropriate notions of quasirandomness in hypergraphs.
For a more in-depth exposition, we refer the reader to the recent survey~\cite{QuasiSurvey}.

\subsection{Basic definitions and notation}

In order to simplify notation and prevent the cluttering of estimates with negligible error terms, we will work with (simple, undirected) hypergraphs \emph{with loops}.
More formally:

\begin{definition}[Hypergraphs]
    For an integer $k\geq 2$, a \emph{$k$-uniform hypergraph} (also called a \emph{$k$-graph}) is a pair $H = (V(H),\, E(H))$ where~$E(H)$ is a symmetric subset of~$V(H)^k$.
    We write $v(H) = |V(H)|$ for its number of vertices, $e(H) = |E(H)|$ for its number of edges, $\delta(H) = e(H)/v(H)^k$ for its edge density, and $H(x_1, \dots, x_k)$ for its edge indicator function $\mathbbm{1}\big[\{x_1, \dots, x_k\} \in E(H)\big]$.
\end{definition}

Note that, since there are at most $\binom{k}{2} v(H)^{k-1}$ tuples in $V(H)^k$ with a repeated element and $v(H)^k$ $k$-tuples in total, the number of possible loops in a large hypergraph is negligible compared to the number of possible edges.
There is no meaningful distinction in our results between hypergraphs with and without loops.

We use the standard expectation notation $\Exp_{x\in X}$ to denote the average $|X|^{-1} \sum_{x\in X}$ over a finite set~$X$.
We use the same notation to denote a set and its indicator function, and write expressions of the form $x = a\pm b$ to denote $a-b \leq x\leq a+b$.
Given a finite set~$A$, a tuple $\xbf = (x_i)_{i \in A}$ indexed by the elements of~$A$ and a subset $B \subseteq A$, we denote by $\xbf_B := (x_j)_{j \in B}$ the projection of~$\xbf$ onto its $B$-coordinates.

\subsubsection*{Polynomial equivalence}

The notion of equivalence that we use in this paper is called \emph{polynomial equivalence}.
Suppose we have two properties $P_1 = P_1(c_1)$ and $P_2 = P_2(c_2)$ which a given object $H$ might satisfy, where each property $P_i$ involves a positive constant $0 < c_i \leq 1$ representing some allowed error bound.
We say that $P_1$ and $P_2$ are \emph{polynomially equivalent} if there exists a constant $K > 0$ such that the following holds for all $\varepsilon \in (0, 1]$:
\begin{itemize}
    \item If $H$ has size at least $K/\varepsilon^K$ and satisfies $P_1(c_1)$ with constant $c_1 \leq \varepsilon^K/K$, then it must also satisfy $P_2(c_2)$ with constant $c_2 = \varepsilon$;
    \item If $H$ has size at least $K/\varepsilon^K$ and satisfies $P_2(c_2)$ with constant $c_2 \leq \varepsilon^K/K$, then it must also satisfy $P_1(c_1)$ with constant $c_1 = \varepsilon$.
\end{itemize}
A sequence of properties $P_1, \dots, P_k$ are polynomially equivalent if they are pairwise polynomially equivalent.

\subsection{Hypergraph quasirandomness}
\label{sec:hyp_quasi}



For every integer $1\leq d < k$ there is an associated equivalence class of notions of quasirandomness for $k$-graphs, which are roughly related to the \emph{lack of correlation} with structures of order~$d$.
Below we will examine several of these quasirandom notions.

\subsubsection*{Discrepancy}

The central notion of hypergraph quasirandomness for us will be related to the \emph{discrepancy} of its edge distribution along cuts of a given order.
Discrepancy is a measure of how far from uniformly distributed the edges of a hypergraph are, and can be quantified using the cut norm as defined below:

\begin{definition}[Cut norm]
Let $k, d \geq 1$ be integers with $d < k$, and let~$V$ be a finite set.
We define the \emph{$d$-cut norm} of a function $f: V^k \to \R$ by
$$\|f\|_{\square^k_d} := \max_{S_B \subseteq V^{B}\; \forall B \in \binom{[k]}{d}} \Bigg| \Exp_{\xbf \in V^{[k]}} \Bigg[ f(\xbf) \prod_{B \in \binom{[k]}{d}}{S_B(\xbf_{B})} \Bigg] \Bigg|,$$
where the maximum is over all collections of sets $(S_B)_{B \in \binom{[k]}{d}}$ where each $S_B$ is a subset of $V^B$.
\end{definition}

Intuitively, the more uniformly distributed the edges of a hypergraph are, the smaller the value of $\|H - \delta(H)\|_{\square^k_d}$ is.
Since the edges of a random hypergraph are typically very uniformly distributed, we can regard having small cut norm as being a quasirandom property for hypergraphs.
More precisely:

\begin{definition}[Discrepancy]
    The $d$-discrepancy of a $k$-graph~$H$ is defined by $$\disc_d(H) := \|H - \delta(H)\|_{\square^k_d} \quad \text{where $\delta(H) = e(H)/v(H)^k$.}$$
    Given $\varepsilon>0$, we say that a $k$-graph~$H$ is \emph{$\varepsilon$-quasirandom of order $d$} if $\disc_d(H) \leq \varepsilon$.
\end{definition}

It is easy to see from the definition of cut norm that
$$0\leq \|f\|_{\square^k_1} \leq \|f\|_{\square^k_2} \leq \dots \leq \|f\|_{\square^k_{k-1}} \leq \|f\|_{\infty},$$
and so
$$0\leq \disc_1(H) \leq \disc_2(H) \leq \dots \leq \disc_{k-1}(H) \leq 1.$$
A $k$-graph which is $\varepsilon$-quasirandom of order~$d$ will thus also be $\varepsilon$-quasirandom of order~$\ell$ for all $1\leq \ell\leq d$.

\begin{remark}
The notion of discrepancy used in the works of Chung~\cite{QuasiClasses} and Kohayakawa, R\"odl and Skokan~\cite{HypergraphQuasirandomnessRegularity} was slightly different from the one given above;
we recall it bellow, and call it \emph{clique discrepancy}.
Given a $d$-graph~$G$, let $\mathcal{K}_k(G)$ denote the set of $k$-cliques in~$G$ (i.e. the collection of $k$-sets of vertices whose $d$-subsets are all edges of~$G$).
The $d$-clique discrepancy of a $k$-uniform hypergraph $H$ is defined as
$$\frac{1}{v(H)^k} \max_{d\text{-graph } G} \big| |H \cap \mathcal{K}_k(G)| - \delta(H) |\mathcal{K}_k(G)| \big|,$$
where the maximum is over all $d$-graphs $G$ on the same vertex set as~$H$.
This notion is formally very similar to our measure $\|H - \delta(H)\|_{\square^k_d}$ of quasirandomness of order~$d$ (once one unpacks all the notation), and
one can show that these two quantities are polynomially related to each other.
We regard $d$-clique discrepancy and $d$-discrepancy as being the same notion in different guises, and have chosen to use the latter for technical reasons.
\end{remark}

\subsubsection*{Counting subhypergraphs}

An important statistic to have on a large hypergraph~$H$ is the number of various smaller hypergraphs~$F$ occurring as a subgraph.
A convenient way of counting such copies is given by the \emph{homomorphism density}:

\begin{definition}[Homomorphism density]
    Let~$F$ and~$H$ be two $k$-graphs.
    The homomorphism density of~$F$ in~$H$, denoted $t(F, H)$, is the probability that a randomly selected map $\phi: V(F)\to V(H)$ preserves edges.
\end{definition}

One can equivalently define the homomorphism density by the formula
$$t(F, H) = \Exp_{\xbf\in V(H)^{V(F)}} \prod_{e\in E(F)} H(\xbf_e),$$
which also makes sense when~$H$ is edge-weighted;
this weighted case will also be used later on.
Note that
$$N_F(H) = t(F, H) v(H)^{v(F)} \pm \binom{v(F)}{2} v(H)^{v(F)-1},$$
where $N_F(H)$ is the total number of labelled subgraphs of~$H$ which are isomorphic to~$F$;
one can thus translate statements about subgraph counts in large hypergraphs to statements about homomorphism densities, and vice versa.

\smallskip

An important class of hypergraphs in our results is the following:

\begin{definition}[$d$-linear hypergraphs]
    A hypergraph~$H$ is said to be $d$-linear if every two edges of~$H$ intersect in at most~$d$ vertices.
    We denote the set of all $d$-linear $k$-graphs by $\mathcal{L}^{(k)}_d$.
\end{definition}

It was proven by Towsner that quasirandomness of degree~$d$ is necessary and sufficient to controls the count of every $d$-linear hypergraph;
see Theorem~\ref{Towsnerthm} below.

\subsubsection*{Deviation}

We recall that the \emph{$k$-octahedron} $\oct^{(k)}$ is the complete $k$-partite $k$-graph where each vertex class has two vertices.
They are generalized by the \emph{squashed octahedra}, defined as follows:

\begin{definition}[Squashed octahedra]
Given integers $1\leq d< k$, we define the squashed octahedron $\oct^{(k)}_{d}$ as the $k$-graph on vertex set $\{x^{(0)}_1,\, x^{(1)}_1,\, \dots,\, x^{(0)}_{d},\, x^{(1)}_{d},\, y_{d+1},\, \dots,\, y_{k}\}$ whose edge set is given by
$$E(\oct^{(k)}_{d}) = \Big\{ \big\{x^{(\omega_1)}_1,\, \dots,\, x^{(\omega_{d})}_{d},\, y_{d+1},\, \dots,\, y_{k}\big\}:\, \omega \in \{0, 1\}^d \Big\}.$$
\end{definition}

Following Chung~\cite{QuasiClasses}, we define the \emph{$d$-deviation} of a $k$-graph~$H$ by
$$\dev_{d}(H) = \mathbb{E}_{\xbf^{(0)}, \xbf^{(1)} \in V(H)^{d}} \mathbb{E}_{y_{d+1}, \dots, y_k \in V(H)} \prod_{\omega \in \{0, 1\}^{d}} \big( H(x_1^{(\omega_1)}, \dots, x_{d}^{(\omega_{d})}, y_{d+1}, \dots, y_k) - \delta(H) \big).$$
Note that this equals the weighted count $t\big(\oct^{(k)}_d,\, H-\delta(H)\big)$ of squashed octahedra.

\subsubsection*{The octahedral norms}

The octahedral norms give an alternative measure for strong quasirandomness, based on a weighted count of octahedra.
Their definition is essentially due to Gowers~\cite{HypergraphRegularityGowers}.

\begin{definition}[Octahedral norm]
Given a function $f: V^k \rightarrow \R$, we define its \emph{octahedral norm} by
\begin{equation} \label{eq:octnorm}
    \|f\|_{\oct^k} := \Bigg( \Exp_{\xbf^{(0)},\, \xbf^{(1)} \in V^k} \prod_{\omega \in \{0, 1\}^k} f \big( \xbf^{(\omega)} \big) \Bigg)^{1/2^k},
\end{equation}
where we write $\xbf^{(\omega)} := \big(x_i^{(\omega_i)}\big)_{i \in [k]}$.
\end{definition}

One can show that the expectation on the right-hand side of~\eqref{eq:octnorm} is nonnegative for every real function~$f$, and that~$\|\cdot\|_{\oct^k}$ indeed defines a norm.
An important property of the octahedral norm is that it has an associated \emph{generalized inner product}, denoted $\langle \cdot \rangle_{\oct^k}$, which we define for $2^k$ functions $f_{\omega}: V^k \rightarrow \R$, $\omega \in \{0, 1\}^k$, by
\begin{equation} \label{OctInnerProduct}
    \big\langle (f_{\omega})_{\omega \in \{0, 1\}^k} \big\rangle_{\oct^k}
    := \Exp_{\xbf^{(0)}, \xbf^{(1)} \in V^k} \prod_{\omega \in \{0, 1\}^k} f_{\omega}\big(\xbf^{(\omega)}\big).
\end{equation}
With this inner product we have that
$\|f\|_{\oct^k}^{2^k} = \big\langle f, f, \dots, f \big\rangle_{\oct^k}$.

A very useful property of the octahedral norms and associated inner products is that they satisfy a type of Cauchy-Schwarz inequality.
This result was first established by Gowers (though with a different notation), and is now known as the \emph{Gowers-Cauchy-Schwarz inequality}:

\begin{lem}[Gowers-Cauchy-Schwarz inequality] \label{GowersCS}
For any collection of functions $f_{\omega}: V^k \rightarrow \R$, $\omega \in \{0, 1\}^k$, we have
$$\big\langle (f_{\omega})_{\omega \in \{0, 1\}^k} \big\rangle_{\oct^k} \leq \prod_{\omega \in \{0, 1\}^k} \| f_{\omega} \|_{\oct^k}.$$
\end{lem}

This lemma is proven via repeated applications of the Cauchy-Schwarz inequality;
see e.g. \cite[Section~4.5]{QuasiSurvey} for a proof.
As a consequence of the Gowers-Cauchy-Schwarz inequality, one can easily show that the octahedral norms are \emph{stronger} than the cut norms:

\begin{lem} \label{lem:cut<oct}
For any function $f: V^k \rightarrow \R$, we have $\| f \|_{\square^{k}_{k-1}} \leq \| f \|_{\oct^k}$.
\end{lem}

\begin{proof}
Given functions $u_B: V^B \rightarrow [0, 1]$, $B \in \binom{[k]}{k-1}$, let $f_{\omega_B}: V^{[k]} \rightarrow \R$ be the function defined by $f_{\omega_B}(\xbf_{[k]}) = u_B(\xbf_B)$, where $\omega_B \in \{0, 1\}^{[k]}$ is the indicator vector of the set $B$.
Denote also $f_{\mathbf{1}} = f$ and $f_{\omega} \equiv 1$ for all $\omega \in \{0, 1\}^{[k]} \setminus \{\mathbf{1}\}$ not contained in the set $\big\{\omega_B: B \in \binom{[k]}{k-1}\big\}$.

Using the Gowers-Cauchy-Schwarz inequality we conclude that
\begin{align*}
    \Bigg| \Exp_{\xbf \in V^{[k]}} \Bigg[ f(\xbf) \prod_{B \in \binom{[k]}{k-1}}{u_{B}(\xbf_{B})} \Bigg] \Bigg| &= \Bigg|\Exp_{\xbf^{(0)}, \xbf^{(1)} \in V^k} \prod_{\omega \in \{0, 1\}^k} f_{\omega} \big(\xbf^{(\omega)}\big) \Bigg| \\
    &\leq \prod_{\omega \in \{0, 1\}^k} \| f_{\omega} \|_{\oct^k}.
\end{align*}
Since clearly $\| f_{\omega} \|_{\oct^k} \leq \| f_{\omega} \|_{\infty} \leq 1$ for all $\omega \in \{0, 1\}^{[k]} \setminus \{\mathbf{1}\}$, the last product is at most $\|f\|_{\oct^k}$.
As this inequality is valid for all functions $u_B: V^B \rightarrow [0, 1]$, $B \in \binom{[k]}{k-1}$, the claim follows.
\end{proof}

\subsubsection*{The hypergraph Equivalence Theorem}

We will next present Towsner's theorem relating multiple notions of quasirandomness for any given order~$d\geq 1$.
We start by constructing the hypergraphs which are complete for these notions.

Given a $k$-partite $k$-graph $F$ with vertex partition $X_1, \dots, X_k$ and a $d$-set of indices $I \in \binom{[k]}{d}$, we define the \emph{$I$-doubling} of $F$ to be the hypergraph $\db_I(F)$ obtained by taking two copies of $F$ and identifying the corresponding vertices in the classes $X_i$, for all $i \in I$.
More precisely, the vertex set of the $I$-doubling is
\begin{equation*}
    V(\db_I(F)) = Y_1 \cup \dots \cup Y_k \quad \text{where} \quad Y_i =
    \begin{cases}
        X_i & \text{ if } i \in I, \\
        X_i \times \{0, 1\} & \text{ if } i \notin I 
    \end{cases}
\end{equation*}
and its edge set is the collection of all $k$-sets of the form
$$\{x_i:\, i \in I\} \cup \{(x_j, a):\, j \in [k] \setminus I\},$$
where $a \in \{0, 1\}$ and $\{x_i:\, i \in [k]\}$ is an edge of $F$.
Starting with the $k$-partite hypergraph with $k$ vertices and a single edge, and then applying consecutively $\db_I$ for every $I \in \binom{[k]}{d}$ (in some arbitrary order), we obtain a $k$-graph which we denote by $\M^{(k)}_{d}$.

Below we reproduce a quantitative version of the main result of Towsner~\cite{SigmaAlgebrasHypergraphs};
this version can be obtained via the methods exposed in~\cite{QuasirandomnessHypergraphs}.

\begin{thm}[Equivalence Theorem for quasirandomness of order $d$] \label{Towsnerthm}
Let $1 \leq d < k$ be integers and let $H$ be a $k$-uniform hypergraph with edge density $\delta$.
Then the following properties are polynomially equivalent:
\begin{itemize}
    \item[$(i)$] $H$ has small $d$-discrepancy: \quad
    $\disc_d(H) \leq c_1$.
    \item[$(ii)$] $H$ correctly counts all $d$-linear hypergraphs:
    $$t(F, H) = \delta^{e(F)} \pm e(F) c_2 \quad \forall F \in \mathcal{L}^{(k)}_{d}.$$
    \item[$(iii)$] $H$ has few copies of $M = \M^{(k)}_{d}$: \quad
    $t(M, H) \leq \delta^{e(M)} + c_3$.
    \item[$(iv)$] $H$ has small deviation with respect to $M = \M^{(k)}_{d}$:
    $$\Exp_{\xbf \in V(H)^{V(M)}} \prod_{e \in E(M)} \big( H(\xbf_{e}) - \delta \big) \leq c_4.$$
\end{itemize}
\end{thm}

\subsection{Examples to keep in mind}
\label{sec:examples}

It might be helpful to keep a concrete example in mind for each of the notions of quasirandomness we consider.
We next give such examples by exploiting the well-known pseudorandom properties of quadratic residues.

Let~$p$ be a large prime, and denote by~$Q_p$ the set of quadratic residues modulo~$p$:
$$Q_p = \big\{ x\in \F_p:\: \text{exists $y\in \F_p$ with $x=y^2$} \big\}.$$
This set has size $(p+1)/2$, and it was shown by Fouvry, Kowalski and Michel~\cite{FKM2013} that it is almost as quasirandom as possible, in the sense that
$$\|Q_p - 1/2\|_{U^{d+1}} = O_d(p^{-1/2^{d+1}}) \quad \text{for all $d\geq 1$.}$$
This bound is of the same order as is expected of a random $\{-1, 1\}$-valued function, while for \emph{any} $\{-1, 1\}$-valued function~$f$ on~$\F_p$ we have $\|f\|_{U^{d+1}} \geq p^{-1/2^{d+1}}$.

We define the \emph{Paley $k$-graph} $\Pcal^{(k)} = \Pcal^{(k)}(p)$ as the hypergraph whose vertices are the elements of~$\F_p$, and where $\{x_1, \dots, x_k\}$ is an edge iff $x_1 + \dots + x_k\in Q_p$.
From the properties of quadratic residues, we see that $\Pcal^{(k)}$ has edge density $1/2 + o(1)$ and satisfies
$$\|\Pcal^{(k)} - 1/2\|_{\square^k_{k-1}} \leq \|\Pcal^{(k)} - 1/2\|_{\oct^k} = \|Q_p - 1/2\|_{U^k} = O_k(p^{-1/2^k}),$$
where the inequality follows from Lemma~\ref{lem:cut<oct} and the first equality follows from a simple change of variables in the expression defining the $U^k$-norm
(see Lemma~\ref{RelationUkOctk} in the next section).
This hypergraph was already considered by Chung and Graham~\cite{QuasirandomHypergraphs} as an example of quasirandom hypergraphs.

Now fix some integer $2 \leq d < k$, and let $\Pcal^{(k)}_d = \Pcal^{(k)}_d(p)$ be the hypergraph encoding $k$-cliques in $\Pcal^{(d)}(p)$.
More explicitly, the vertex set of $\Pcal^{(k)}_d$ is $\F_p$ and a $k$-set $\{x_1, \dots, x_k\} \subset \F_p$ is an edge iff
$$\sum_{i \in B} x_i \in Q_p \quad \text{for all } B \in \binom{[k]}{d}.$$
The edge density of $\Pcal^{(k)}_d$ is precisely the homomorphism density of $k$-cliques in $\Pcal^{(d)}$, which (by quasirandomness of $\Pcal^{(d)}$) is equal to $2^{-\binom{k}{d}} + o(1)$.

It is easy to show that $\big\|\Pcal^{(k)}_d - 2^{-\binom{k}{d}} \big\|_{\square^k_{d-1}} = o(1)$, since any witness sets for high $(d-1)$-discrepancy of $\Pcal^{(k)}_d$ can be turned into witness sets for high $(d-1)$-discrepancy of $\Pcal^{(d)}$ (and these cannot exist).
Finally, we note that
\begin{align*}
    \big\|\Pcal^{(k)}_d - 2^{-\binom{k}{d}}\big\|_{\square^k_{d}}
    &\geq \Exp_{\xbf \in \F_p^k} \Bigg[\big(\Pcal^{(k)}_d(\xbf) - 2^{-\binom{k}{d}}\big) \prod_{B\in \binom{[k]}{d}} \Pcal^{(d)}(\xbf_B) \Bigg] \\
    &= \Exp_{\xbf \in \F_p^k} \Big[\big(\Pcal^{(k)}_d(\xbf) - 2^{-\binom{k}{d}}\big) \Pcal^{(k)}_d(\xbf) \Big] \\
    &= 2^{-\binom{k}{d}} - 2^{-2\binom{k}{d}} + o(1).
\end{align*}
It follows that $\Pcal^{(k)}_d$ is quasirandom of order $d-1$, but \emph{not} quasirandom of order~$d$.

\section{Uniform sets and their Cayley hypergraphs} \label{sec:Cayley}

As already noted in the Introduction, a convenient way of considering additive sets and hypergraphs in the same framework is by defining the \emph{Cayley hypergraph} associated to an additive set:

\begin{definition}[Cayley hypergraph]
Let $A$ be a subset of an additive group $G$ and $k\geq 2$ be an integer.
The \emph{Cayley $k$-graph} of~$A$ is the hypergraph~$\Gamma^{(k)}_A$ whose vertices are all elements of~$G$, and where~$k$ vertices $x_1, \dots, x_k$ form an edge iff $x_1 + \dots + x_k \in A$.
\end{definition}



Note that the Paley $k$-graph $\Pcal^{(k)}$ considered in Section~\ref{sec:examples} is an example of a Cayley hypergraph.
The next definition is a technical piece of notation meant to simplify the exposition somewhat:

\begin{definition}[Summing operator]
Given an integer $k$ and an additive group $G$, we denote by $\Sigma: G^k \rightarrow G$ its \emph{summing operator}
$$\Sigma(x_1, x_2, \dots, x_k) := x_1 + x_2 + \dots + x_k.$$
\end{definition}

\begin{remark}
There is a slight abuse of notation here since the same designation is used no matter how many terms are being summed or which additive group the summands belong to.
These hidden parameters may change each time the operator is used.
\end{remark}

With this piece of notation, we can write the indicator function of a Cayley hypergraph~$\Gamma^{k}(A)$ on~$G^k$ more economically as~$A \circ \Sigma$.
We can similarly define a weighted Cayley $k$-graph associated to a function $f: G \rightarrow \R$ by
$$\Gamma^{(k)}_f(x_1, \dots, x_k) := f \circ \Sigma(x_1, \dots, x_k) \quad \text{for all $x_1, \dots, x_k \in G$;}$$
this extension will be helpful for simplifying some expressions.

A simple but important property of our notions of quasirandomness for additive sets and Cayley hypergraphs is their \emph{translation invariance}:

\begin{definition}[Translation operator]
Given an element $a \in G$, we define the \emph{translation operator} $\T^{a}$ on $\R^G$ by $\T^{a} f(x) := f(x+a)$.
If $A$ is (the indicator function of) a set, then $\T^a A$ is (the indicator function of) the translated set $A - a$.
\end{definition}

We will repeatedly make use of the easily-proven identities $\|\T^a f\|_{U^k} = \|f\|_{U^k}$ and
$\big\|\Gamma^{(k)}_{\T^{a}f}\big\|_{\square^k_d} = \big\|\Gamma^{(k)}_f\big\|_{\square^k_d}$ for all $1 \leq d < k$.
This last identity intuitively means that the translation operation preserves the cut structure of Cayley hypergraphs, and allows us to analyze those hypergraphs by more `arithmetical' means.

As a first step towards connecting the notions of quasirandomness in additive groups and hypergraphs, we give an easy (and well-known) connection between the $U^k$ uniformity norms and the $\oct^k$ octahedral norms.

\begin{lem}[Relationship between the $U^k$ and $\oct^{k}$ norms] \label{RelationUkOctk}
For every real function $f: G \rightarrow \R$ we have that $\| \Gamma^{(k)}_f \|_{\oct^{k}} = \| f \|_{U^k}$.
\end{lem}

\begin{proof}
We make the change of variables
$$x := \Sigma \big(\xbf^{(0)}\big) = x_1^{(0)} + \dots + x_k^{(0)}, \quad h_i := x^{(1)}_i - x^{(0)}_i \quad \text{for $i\in [k]$.}$$
Then $\Sigma \big(\xbf^{(\omega)}\big) = x + \sum_{i=1}^k \omega_i h_i$ for all $\omega \in \{0, 1\}^k$, and the identity follows.
\end{proof}

\subsection{Connecting uniformity and quasirandomness}

Our main technical result is that uniformity of a given degree~$d$ for a set~$A$ is polynomially equivalent to quasirandomness of the same order~$d$ for its Cayley hypergraph~$\Gamma^{(k)}_A$, for any value of $k>d$.
This generalizes a theorem of Aigner-Horev and H\`an~\cite{LinearQuasirandomness}, which considers the special case where~$d=1$.

Recall that a $k$-graph~$H$ is $\varepsilon$-quasirandom of order $d$ if $\|H - \delta(H)\|_{\square^k_d} \leq \varepsilon$,
and a set $A\subseteq G$ is $\varepsilon$-uniform of degree~$d$ if $\|A - \delta\|_{U^{d+1}} \leq \varepsilon$ where $\delta = |A|/|G|$.
Our main result is the following:

\begin{thm}[Theorem~\ref{thm:mainintro} restated] \label{thm:main}
Let $d \geq 1$ be an integer and $\varepsilon\in (0, 1)$.
Let $G$ be a finite additive group and $A \subseteq G$ be a subset.
\begin{itemize}
    \item[$(i)$] If $A$ is $\varepsilon$-uniform of degree $d$, then for all $k \geq d + 1$ the Cayley hypergraph $\Gamma^{(k)}_A$ is $\varepsilon$-quasirandom of order $d$.
    \item[$(ii)$] Conversely, if $\Gamma^{(k)}_A$ is $\varepsilon$-quasirandom of order $d$ for some $k \geq d+1$, then $A$ is $2\varepsilon^{c_{k, d}}$-uniform of degree $d$.
    Here we can take $c_{k, d} = 2^{-(d+2)(2d+2)^k}$.
\end{itemize}
\end{thm}

As might be expected from the bounds given in this statement, the proof of proposition~$(ii)$ is much more involved than that of proposition~$(i)$.
Both proofs are elementary in the sense that they use only the triangle inequality and Cauchy-Schwarz, but the applications of Cauchy-Schwarz needed for proving proposition~$(ii)$ are somewhat intricate and require a careful analysis.

\begin{proof}[Proof of proposition~$(i)$]
Let $f_A := A-\delta$ be the balanced indicator function of the considered set~$A$.
Choose optimal functions $u_B: G^B \rightarrow [0, 1]$, $B \in \binom{[k]}{d}$, so that
$$\big\|\Gamma^{(k)}_A - \delta \big\|_{\square^{k}_{d}} = \Bigg| \Exp_{\xbf \in G^{k}} \Bigg[ f_A(\Sigma(\xbf)) \prod_{B \in \binom{[k]}{d}} u_B(\xbf_B) \Bigg] \Bigg|.$$
We may separate the first $d+1$ variables $\xbf_{[d+1]}$ from the rest and write
\begin{equation*}
    \big\|\Gamma^{(k)}_A - \delta\big\|_{\square^{k}_{d}}
    = \Bigg| \Exp_{\xbf_{[k] \setminus [d+1]}} \Exp_{\xbf_{[d+1]}} \Bigg[ f_A\big(\Sigma(\xbf_{[d+1]}) + \Sigma(\xbf_{[k] \setminus [d+1]})\big) \prod_{B \in \binom{[k]}{d}} u_B(\xbf_B) \Bigg] \Bigg|,
\end{equation*}
where the first expectation is over $G^{[k] \setminus [d+1]}$ and the second is over $G^{[d+1]}$.

Now we fix $\xbf_{[k] \setminus [d+1]} \in G^{[k] \setminus [d+1]}$ and consider the inner expectation in the last expression.
Writing $y := \Sigma(\xbf_{[k] \setminus [d+1]})$, this expression can be written as
$$\Exp_{\xbf_{[d+1]}} \Bigg[ \T^{y}f_A \circ \Sigma(\xbf_{[d+1]}) \prod_{D \in \binom{[d+1]}{d}} v_{D}(\xbf_{D}) \Bigg]$$
for some suitable functions $v_{D}: G^D \rightarrow [0, 1]$, $D \in \binom{[d+1]}{d}$, and thus has absolute value at most
$$\| \T^{y}f_A \circ \Sigma \|_{\square^{d+1}_{d}} = \big\| \Gamma^{(d+1)}_{\T^{y}A} - \delta\big\|_{\square^{d+1}_{d}} = \big\| \Gamma^{(d+1)}_A - \delta\big\|_{\square^{d+1}_{d}}.$$
Since the octahedral norm is stronger than the cut norm (Lemma~\ref{lem:cut<oct}), this last term is at most
$\big\|\Gamma^{(d+1)}_A - \delta\big\|_{\oct^{d+1}} = \|A - \delta\|_{U^{d+1}}$
(where we used Lemma~\ref{RelationUkOctk}).
Averaging over $\xbf_{[k] \setminus [d+1]} \in G^{[k] \setminus [d+1]}$ and using the triangle inequality we conclude that
$\big\|\Gamma^{(k)}_A - \delta\big\|_{\square^{k}_{d}} \leq \|A - \delta\|_{U^{d+1}} \leq \varepsilon,$
as wished.
\end{proof}

The rest of this section will be devoted to the proof of proposition~$(ii)$.

\subsection{Proof that quasirandomness implies uniformity}

Fix an additive group~$G$ and integers~$k$, $d\geq 1$ with~$k\geq d+1$.
We wish to show that $A\subseteq G$ is uniform of degree~$d$ whenever~$\Gamma^{(k)}_A$ is quasirandom of order~$d$.
Our proof will proceed via an iterative argument, where we construct and analyze several systems of linear forms defined on~$G$.

\subsubsection{Linear systems and norms}

We will only consider linear forms $\phi: G^{V} \to G$ whose coefficients are either~0 or~1, where~$V$ is a finite index set for the variables.
These forms can be characterized by their \emph{support}:
this is the subset $\supp(\phi) \subseteq V$ such that
$$\phi(\xbf) = \sum_{v\in \supp(\phi)} x_v \quad \text{for all $\xbf\in G^{V}$.}$$
A collection of linear forms $\Phi = \{\phi_1, \dots, \phi_m\}$ is a \emph{linear system}, and its support is the union of the supports of all its constituent forms:
$\supp(\Phi) = \bigcup_{i=1}^m \supp(\phi_i)$.

Following Green and Tao~\cite{LinearEquationsPrimes}, we say that a linear system $\Phi = \{\phi_1, \dots, \phi_m\}$ is in \emph{$s$-normal form} if, for every~$i\in [m]$, there is a subset $\sigma_i \subseteq \supp(\phi_i)$ of size at most~$s+1$ which is not completely contained in~$\supp(\phi_j)$ for any~$j\neq i$.
The importance of this notion is given by the next result, essentially due to Green and Tao;
see \cite[Appendix~C]{LinearEquationsPrimes} or \cite[Section~2]{TrueComplexity} for a proof.

\begin{lem} \label{lem:snormal}
    If~$\Phi = \{\phi_1, \dots, \phi_m\}$ is a linear system in $s$-normal form, then for all functions $f_1, \dots, f_m: G\to [-1, 1]$ we have
    $$\bigg|\Exp_{\xbf\in G^V} \prod_{i=1}^m f_i(\phi_i(\xbf)) \bigg| \leq \min_{1\leq i\leq m} \|f_i\|_{U^{s+1}}.$$
\end{lem}


There is a specific set~$V_0 := \{(0,1), (0,2), \dots, (0,k)\}$ of indices for the variables of our linear forms which is considered separately from the rest, and will play a crucial role in our arguments.
Roughly speaking, this set indexes the~$k$ variables we really care about, while the other variables are only there for helping with the analysis and will be substituted at the end by suitably-chosen values.
Given a linear form~$\phi$, we then define its \emph{weight} $w_0(\phi)$ as the number of variables in~$V_0$ that it uses:
$w_0(\phi) := |\supp(\phi) \cap V_0|$.
The weight of a linear system~$\Phi$ is the maximum weight of one of its forms:
$w_0(\Phi) = \max \{w_0(\phi): \phi\in \Phi\}$.

In our analysis we will need a cut-type seminorm associated to linear systems, which serves to bridge the gap between the $U^{d+1}$-norm for additive sets and the $d$-cut norm for their Cayley hypergraphs.

\begin{definition}[Cut-type norms]
Let~$\Phi$ be a linear system on~$G$, and denote $V(\Phi) = V_0 \cup \supp(\Phi)$.
Given a function $f: G\to \R$, we define
$$\|f\|_{\square(\Phi)} := \max_{u_{\phi}: G \rightarrow [-1,1],\, \forall \phi \in \Phi} \Exp_{\xbf \in G^{V(\Phi)}} \Bigg[ f\bigg( \sum_{v \in V_0}x_v \bigg) \prod_{\phi \in \Phi} u_{\phi}\big( \phi(\xbf)\big) \Bigg].$$
\end{definition}

We can now give an outline of our proof of proposition~$(ii)$.
We will proceed via an iterative algorithm, where at each step~$s$ we have a linear system~$\Phi_s$ characterized by the support of its linear forms.
We start with~$\Phi_0$ containing only the linear form whose support is~$V_0$, which then has weight $k\geq d+1$.
Whenever the system~$\Phi_s$ in consideration has some form~$\phi$ with weight higher than~$d$, we replace~$\phi$ by~$2^{d+1}-1$ `dual' forms of strictly lower weight, thus creating the system~$\Phi_{s+1}$.
The important thing here is that the cut-type norms associated to the systems~$\Phi_s$ and~$\Phi_{s+1}$ are related to each other (by some Cauchy-Schwarz magic trick).
As soon as all forms in~$\Phi_s$ have weight at most~$d$ we stop;
because the weights are bounded by~$d$, we are then able to bound the associated norm~$\|\cdot\|_{\square(\Phi_s)}$ in terms of the usual $d$-cut norm.
Moreover, we show that the first norm~$\|\cdot\|_{\square(\Phi_1)}$ is bounded from \emph{below} by the $U^{d+1}$ norm.
The proposition then follows by applying the resulting norm inequalities to the balanced indicator function~$A-\delta$ of the considered set~$A\subseteq G$.

\subsubsection{Dual linear forms}

It remains to give the notion of dual linear forms to be used in our algorithm, which is motivated by the notion of \emph{$U^{d+1}$-dual function}.
For a function $f: G\to \R$, its dual function~$\Dcal_{d+1}f$ is defined by
$$\Dcal_{d+1}f(y) = \Exp_{h_1, \dots, h_{d+1} \in G} \prod_{\omega \in \cubem} f\bigg(y + \sum_{i=1}^{d+1} \omega_i h_i\bigg),$$
so that we can write $\|f\|_{U^{d+1}}^{2^{d+1}} = \Exp_{y\in G} \big[f(y) \Dcal_{d+1}f(y)\big]$.
The dual linear forms are meant to emulate this notion, but relative a given `heavy' form~$\phi$.

Let $\phi: G^V \to G$ be a linear form of weight $w_0(\phi) \geq d+1$, and let~$I_d$ be set of $d+1$ elements which is disjoint from~$V$.
For each $\omega\in \cubem$, we define a linear form $\Dfrak^{\omega}\phi: G^{V\cup I_d} \to G$ with $w_0(\Dfrak^{\omega}\phi) < w_0(\phi)$
such that
\begin{equation} \label{eq:dual_prop}
    \Dcal_{d+1}f(\phi(\xbf_V)) = \Exp_{\xbf_{I_d} \in G^{I_d}} \prod_{\omega \in \cubem} f\big(\Dfrak^{\omega}\phi(\xbf_{V\cup I_d})\big) \quad \text{for all } \xbf_V \in G^V.
\end{equation}
This can be achieved in the following way:

\begin{definition}[Dual forms]
    Given a linear form $\phi$ with weight $w_0(\phi) \geq d+1$,
    let~$I_d$ be a copy of the set~$\{1, 2, \dots, d+1\}$ which is disjoint from~$\supp(\phi)\cup V_0$, and write
    $$\supp(\phi) \cap V_0 = \big\{(0, j_1), (0, j_2), \dots, (0, j_{w_0(\phi)})\big\}.$$
    For every $\omega\in \cubem$, we define a linear form~$\Dfrak^{\omega}\phi$ (with coefficients either~0 or~1) by
    \begin{align*}
        \supp(\phi) \setminus \supp(\Dfrak^{\omega}\phi) &= \big\{(0, j_i):\: 1\leq i\leq d+1,\, \omega_i = 1\big\}, \\
        \supp(\Dfrak^{\omega}\phi) \setminus \supp(\phi) &= \big\{i\in I_d:\: 1\leq i\leq d+1,\, \omega_i = 1\big\}.
    \end{align*}
\end{definition}

In other words, $\Dfrak^{\omega}\phi$ is constructed from~$\phi$ by substituting the variables indexed by those~$(0, j_i)$ with $\omega_i = 1$ by new variables.
Note that it does not matter in which order we choose to label the elements in either $\supp(\phi) \cap V_0$ or~$I_d$, as the obtained forms will be equivalent for any labeling.
Note also that this definition indeed satisfies equation~\eqref{eq:dual_prop}:
writing~$V$ for the support of~$\phi$ and performing the change of variables
$$y = \sum_{v\in V} x_v \quad \text{and} \quad h_i = x_i - x_{(0, j_i)} \quad \text{for all } 1\leq i\leq d+1,$$
we see that
$\Dfrak^{\omega}\phi(\xbf_{V\cup I_d}) = y + \sum_{i=1}^{d+1} \omega_i h_i$.

\subsubsection{The main algorithm and its analysis}

Consider the following algorithm:\footnote{The symbol~`$\gets$' means `gets', which is essentially an equality sign that also allows for expressions of the form $s\gets s+1$ (meaning that variable~$s$ gets increased by 1 at this point).}

\medskip
\noindent\textbf{Algorithm \texttt{SystemCut$(k, d)$}}
\begin{algorithmic}
    \State $V_0 \gets \big\{(0,1), (0,2), \dots, (0,k)\big\}$
    \State $\Phi_0 \gets \big\{ \Sigma_{V_0} \big\}$
    \State $s \gets 0$
    \While{$w_0(\Phi_s) > d$} 
        \State take $\psi_s \in \Phi_s$ with $w_0(\psi_s) = w_0(\Phi_s)$ 
        \State $\Phi_{s+1} \gets \big(\Phi_s \setminus \{\psi_s\}\big) \cup \big\{ \Dfrak^{\omega}\psi_s:\, \omega \in \cubem \big\}$
        \State $s \gets s+1$
    \EndWhile
    \State $\mathfrak{s_f} \gets s$
\end{algorithmic}
\medskip

Regarding this algorithm, we will show:

\begin{lem} \label{lem:properties}
For any bounded function $f: G \rightarrow [-1, 1]$ we have:
\begin{enumerate}
    \item $\|f\|_{\square(\Phi_1)} \geq \|f\|_{U^{d+1}}^{2^{d+1}}$.
    \item $\|f\|_{\square(\Phi_{\sf})} \leq 2^{\binom{k}{d}} \|\Gamma^{(k)}_f\|_{\square^k_d}$.
    \item $\Phi_s \cup \{\Sigma_{V_0}\}$ is in $d$-normal form for all $s \geq 1$.
    \item $\|f\|_{\square(\Phi_{s+1})} \geq \|f\|_{\square(\Phi_s)}^{2^{d+2}}$ for $1 \leq s < \sf$.
    \item $\sf < (2d+2)^k$.
\end{enumerate}
\end{lem}

With the help of this lemma, proposition~$(ii)$ of Theorem~\ref{thm:main} easily follows:

\begin{proof}[Proof of Theorem~\ref{thm:main}.$(ii)$]
Suppose~$\Gamma^{(k)}_A$ is $\varepsilon$-quasirandom of order~$d$.
By \emph{Item~2} of Lemma~\ref{lem:properties} applied to $f = A-\delta$, we obtain
$\|A - \delta\|_{\square(\Phi_{\sf})} \leq 2^{\binom{k}{d}} \varepsilon$.
Applying \emph{Item~4} recursively from $s = \sf-1$ to $s=1$, we conclude that
$\|A - \delta\|_{\square(\Phi_1)}^{2^{(d+2)\sf}} \leq 2^{\binom{k}{d}} \varepsilon$.
Using \emph{Item~1} together with the bound $\sf < (2d+2)^k$ from \emph{Item~5}, we deduce that
$$\|A - \delta\|_{U^{d+1}}^{2^{(d+2)(2d+2)^k}}
\leq \big(\|A - \delta\|_{U^{d+1}}^{2^{d+1}}\big)^{2^{(d+2)\sf}}
\leq \|A - \delta\|_{\square(\Phi_1)}^{2^{(d+2)\sf}} \leq 2^{\binom{k}{d}} \varepsilon,$$
and thus $A$ is $2\varepsilon^{c_{k, d}}$-uniform of degree $d$ for $c_{k, d} = 2^{-(d+2)(2d+2)^k}$.
\end{proof}

It then suffices to prove Lemma~\ref{lem:properties}, which we do next.

\begin{proof}[Proof of Lemma~\ref{lem:properties}]
Throughout this proof we will denote the support of the linear system~$\Phi_s$ by~$V_s$,
that is: $V_s = \bigcup_{\phi\in \Phi_s} \supp(\phi)$.
Note that $V_0 \subset V_1 \subset \dots \subset V_\sf$.

\medskip
\emph{Item 1.}
Using identity~\eqref{eq:dual_prop} (which motivated our definition of the dual forms), we have that
\begin{align*}
    \|f\|_{U^{d+1}}^{2^{d+1}}
    &= \Exp_{\xbf\in G^{V_0}} \big[f(\Sigma_{V_0}(\xbf)) \Dcal_{d+1}f(\Sigma_{V_0}(\xbf))\big] \\
    &= \Exp_{\xbf \in G^{V_0\cup I_d}}\Bigg[ f\bigg( \sum_{v\in V_0} x_v \bigg) \prod_{\omega \in \cubem} f\big(\Dfrak^{\omega}\Sigma_{V_0}(\xbf)\big) \Bigg] \\    
    &= \Exp_{\xbf \in G^{V_1}} \Bigg[ f\bigg( \sum_{v\in V_0} x_v \bigg) \prod_{\phi \in \Phi_1} f\big( \phi(\xbf)\big) \Bigg].
\end{align*}
Since $\|f\|_{\infty} \leq 1$, this last expression is at most~$\|f\|_{\square(\Phi_1)}$, as wished.

\medskip
\emph{Item 2.}
By the definition of~$\sf$ we have that $w_0(\phi) \leq d$ for all $\phi\in \Phi_{\sf}$.
Choose optimal functions $u_{\phi}: G \rightarrow [-1,1]$, $\phi \in \Phi_{\sf}$, such that
$$\|f\|_{\square(\Phi_\sf)} = \Exp_{\xbf \in G^{V_\sf}} \Bigg[ f\bigg(\sum_{v\in V_0} x_v\bigg) \prod_{\phi \in \Phi_\sf} u_{\phi}\big(\phi(\xbf)\big) \Bigg].$$
By the averaging principle, we can fix the variables indexed by $V_\sf \setminus V_0$ to be some $\ybf_{V_\sf \setminus V_0} \in G^{V_\sf \setminus V_0}$ for which
$$\|f\|_{\square(\Phi_\sf)} \leq \Exp_{\xbf_{V_0} \in G^{V_0}} \Bigg[ f\bigg(\sum_{v\in V_0} x_v\bigg) \prod_{\phi \in \Phi_\sf} u_{\phi} \big(\phi\big(\xbf_{V_0}, \ybf_{V_\sf \setminus V_0}\big)\big) \Bigg].$$

Note that each function in the product above depends on at most~$d$ of the variables~$\xbf_{V_0}$, and so we can write it as~$h_B(\xbf_B)$ for some set $B\subset V_0$ of size at most~$d$ and some function $h_B: G^B \to [-1, 1]$.
It follows that
$$\|f\|_{\square(\Phi_\sf)} \leq \max_{h_B: G^B \to [-1, 1],\, \forall B\in \binom{V_0}{d}} \Exp_{\xbf_{V_0} \in G^{V_0}} \Bigg[ f\bigg(\sum_{v\in V_0} x_v\bigg) \prod_{B\in \binom{V_0}{d}} h_B(\xbf_B) \Bigg].$$
Decompose each function $h_B$ in the expression above into its positive part $h_B' := \max\{h_B, 0\}$ and negative part $h_B'' := \max\{-h_B, 0\}$, thus writing $h_B = h_B' - h_B''$.
Expanding the resulting product into $2^{\binom{k}{d}}$ terms and using the triangle inequality, we conclude that this expression is bounded by
$$2^{\binom{k}{d}} \max_{g_B: G^B \to [0, 1],\, \forall B\in \binom{V_0}{d}} \Bigg|\Exp_{\xbf_{V_0} \in G^{V_0}} \Bigg[ f\bigg(\sum_{v\in V_0} x_v\bigg) \prod_{B\in \binom{V_0}{d}} g_B(\xbf_B) \Bigg]\Bigg|.$$
This is precisely $2^{\binom{k}{d}} \|\Gamma^{(k)}_f\|_{\square^k_d}$, finishing the proof.

\medskip
\emph{Item 3.}
Fix~$1\leq s \leq \sf$, and let $\phi\in \Phi_s$ be any form.
By construction, we have that~$\phi = \Dfrak^{\omega}\psi_t$ for some $\omega\in \cubem$ and some~$0\leq t < s$
(where $\psi_t\in \Phi_t$ is the heavy form chosen by the algorithm at step~$t$).
Denote the copy of~$I_d = [d+1]$ used at step~$t$ by
$\{(t+1, 1), \dots, (t+1, d+1)\}$,
and write
$$\supp(\psi_t)\cap V_0 = \big\{(0, j_1), (0, j_2), \dots, (0, j_{w_0(\psi_t)})\big\}.$$
It is then easy to check that $\Dfrak^{\omega}\psi_t$ is the only form in $\Phi_s \cup \{\Sigma_{V_0}\}$ whose support contains
$$\big\{ (0, j_i):\: 1\leq i\leq d+1,\, \omega_i = 0\big\} \cup \big\{ (t+1, i):\: 1\leq i\leq d+1,\, \omega_i = 1\big\}.$$
Finally, $\Sigma_{V_0}$ is the only form in $\Phi_s \cup \{\Sigma_{V_0}\}$ which uses all of the variables $(0, 1), \dots$, $(0, d+1)$, and so this system is in $d$-normal form.

\medskip
\emph{Item 4.}
Choose optimal functions $u_{\phi}: G \rightarrow [-1,1]$, $\phi \in \Phi_s$, such that
$$\|f\|_{\square(\Phi_s)} = \Exp_{\xbf \in G^{V_s}} \Bigg[ f\big( \Sigma_{V_0}(\xbf) \big) \prod_{\phi \in \Phi_s} u_{\phi}\big(\phi(\xbf)\big) \Bigg].$$
We shift our focus to the function $u_{\psi}$,
where $\psi = \psi_s \in \Phi_s$ is the linear form of maximal weight chosen by the algorithm at step~$s$.
We can rewrite the expectation above as $\Exp_{\xbf \in G^{V_s}} \big[u_{\psi}(\psi(\xbf)) g(\psi(\xbf))\big]$, where
\begin{equation*} 
    g(z) := \Exp_{\ybf \in G^{V_s}:\, \psi(\ybf)=z}\Bigg[ f\big( \Sigma_{V_0}(\ybf) \big) \prod_{\phi \in \Phi_s \setminus \{\psi\}} u_{\phi}\big(\phi(\ybf)\big) \Bigg].
\end{equation*}
Note that $\|g\|_{\infty} \leq 1$, as all functions appearing in its definition are $1$-bounded.

Since $\|u_{\psi}\|_{\infty} \leq 1$, by Cauchy-Schwarz we have
\begin{align*}
    \|f\|_{\square(\Phi_s)}
    &= \Exp_{\xbf \in G^{V_s}} \big[u_{\psi}(\psi(\xbf)) g(\psi(\xbf))\big] \\
    &\leq \Exp_{\xbf \in G^{V_s}} \big[g(\psi(\xbf))^2\big]^{1/2} \\
    &= \Exp_{\xbf \in G^{V_s}} \Bigg[ g\big(\psi(\xbf)\big) f\big(\Sigma_{V_0}(\xbf)\big) \prod_{\phi \in \Phi_s \setminus \{\psi\}} u_{\phi}\big(\phi(\xbf)\big) \Bigg]^{1/2}.
\end{align*}
Now we use the fact (from \emph{Item~3.}) that $\Phi_s \cup \{\Sigma_{V_0}\}$ is in $d$-normal form.
As all terms inside the expectation above are $1$-bounded, by Lemma~\ref{lem:snormal} we conclude that this last expression is at most $\|g\|_{U^{d+1}}^{1/2}$, and thus
\begin{equation} \label{eq:fandg}
    \|f\|_{\square(\Phi_s)} \leq \|g\|_{U^{d+1}}^{1/2}.
\end{equation}

Next we bound~$\|g\|_{U^{d+1}}$.
Note that
\begin{align*}
    \|g\|_{U^{d+1}}^{2^{d+1}}
    &= \Exp_{\xbf \in G^{V_s}} \big[g(\psi(\xbf)) \Dcal_{d+1}g(\psi(\xbf))\big] \\
    &= \Exp_{\xbf \in G^{V_s}} \Bigg[ f\big( \Sigma_{V_0}(\xbf) \big) \prod_{\phi \in \Phi_s \setminus \{\psi\}} u_{\phi}\big(\phi(\xbf)\big) \cdot \Dcal_{d+1}g(\psi(\xbf)) \Bigg].
\end{align*}
By construction of the linear forms $\Dfrak^{\omega}\psi$, for each fixed $\xbf_{V_s} \in G^{V_s}$ we have
$$\Dcal_{d+1}g(\psi(\xbf_{V_s})) = \Exp_{\xbf_{V_{s+1} \setminus V_s} \in G^{V_{s+1} \setminus V_s}} \prod_{\omega \in \cubem} g\big(\Dfrak^{\omega}\psi(\xbf_{V_{s+1}})\big);$$
it follows that
\begin{equation*}
    \|g\|_{U^{d+1}}^{2^{d+1}} = \Exp_{\xbf \in G^{V_{s+1}}} \Bigg[ f\big( \Sigma_{V_0}(\xbf) \big) \prod_{\phi \in \Phi_s \setminus \{\psi\}} u_{\phi}\big(\phi(\xbf)\big) \prod_{\omega \in \cubem} g\big(\Dfrak^{\omega}\psi(\xbf)\big) \Bigg].
\end{equation*}
Since $\Phi_{s+1} = \big(\Phi_s \setminus \{\psi\}\big) \cup \big\{ \Dfrak^{\omega}\psi:\, \omega \in \cubem \big\}$,
this last expression is at most~$\|f\|_{\square(\Phi_{s+1})}$.
We then conclude by inequality~\eqref{eq:fandg}.

\medskip
\emph{Item 5.}
Denote by~$\sf(n)$ the final value of~$\sf$ in the algorithm \texttt{SystemCut$(n,d)$}
(that is, when~$|V_0| = n$).
We will show by induction that $\sf(n) < (2d+2)^n$ for every $n\geq 1$.

First we note that~$\sf(n)$ equals the number of times we enter the loop in the algorithm;
in particular $\sf(n)=0$ if $n \leq d$, as in this case $w_0(\Sigma_{V_0}) = n\leq d$ and we do not enter the loop.
This takes care of the base case for the induction.

After we enter the loop for the first time, we replace the form~$\Sigma_{V_0}$ of weight~$n$ by the forms $\Dfrak^{\omega}\Sigma_{V_0}$, $\omega\in \cubem$, which have weight
$$w_0(\Dfrak^{\omega}\Sigma_{V_0}) = (d+1 - |\omega|) + (n-d-1) = n - |\omega|.$$
It follows that
$$\sf(n) = 1 + \sum_{\omega \in \cubem} \sf(n - |\omega|) = 1 + \sum_{i=1}^{d+1} \binom{d+1}{i} \sf(n-i).$$
By the induction hypothesis we have:
\begin{align*}
    \sum_{i=1}^{d+1} \binom{d+1}{i} \sf(n-i)
    &< \sum_{i=1}^{d+1} \binom{d+1}{i} (2d+2)^{n-i} \\
    &= (2d+2)^n \sum_{i=1}^{d+1} \binom{d+1}{i} \Big(\frac{1}{2d+2}\Big)^i \\
    &= (2d+2)^n \bigg(\Big(1 + \frac{1}{2d+2}\Big)^{d+1} - 1\bigg).
\end{align*}
Using that $1+x \leq e^x$ for all $x\geq 0$, we conclude that
$$\sf(n) < 1 + (2d+2)^n \big(e^{1/2} - 1\big) < (2d+2)^n,$$
which finishes the proof by induction.
\end{proof}


\section{The Equivalence Theorem for Cayley hypergraphs}
\label{sec:set_hypergraph}

In the previous section we showed that there is a close relationship between the notions of uniformity for additive sets and quasirandomness for hypergraphs.
In particular, if a set $A\subseteq G$ is uniform of degree~$d$, then its Cayley hypergraphs $\Gamma^{(k)}_A$ will all be quasirandom of order~$d$.
It then follows from Towsner's Theorem (Theorem~\ref{Towsnerthm}) that one can count all $d$-linear subgraphs inside Cayley hypergraphs of sets that are uniform of degree~$d$.

Interestingly, the extra symmetries satisfied by Cayley hypergraphs imply that a much \emph{stronger} result is true.
In order to show this we need to define another family of hypergraphs:

\begin{definition}[$s$-simple hypergraphs]
Given $s \geq 1$, we say that a hypergraph $F$ is \emph{$s$-simple} if the following is true:
for every edge $e \in E(F)$, there exists a set of $s$ vertices $\{v_1, \dots, v_{s}\} \subseteq e$ which is not contained in any other edge of $F$
(i.e. $\{v_1, \dots, v_{s}\} \nsubseteq e'$ for all $e' \in E(F) \setminus \{e\}$).
We denote the set of all $s$-simple $k$-graphs by $\mathcal{S}^{(k)}_{s}$.
\end{definition}

It is immediate from the definition that all $d$-linear hypergraphs are $(d+1)$-simple, but as the next example shows the converse is false.

\begin{ex}
Let $F$ be the connected $k$-graph on $2k-d$ vertices and two edges.
Then $F$ is only $d$-linear, but it is $1$-simple.
\end{ex}

This very easy example shows that the difference between `how linear' and `how simple' a hypergraph can be is unbounded.
It also shows that one cannot hope to control the count of all $s$-simple subgraphs by using only quasirandomness of order~$s$ (say):
the Paley graph $\Pcal^{(k)}_{d}$ from Section~\ref{sec:examples} is quasirandom of order~$d-1$, but contains twice as many copies of the 1-simple $k$-graph~$F$ given above than would be expected.


By contrast, it was implicitly shown by Gowers and Wolf \cite[Section~4]{TrueComplexity} that Cayley hypergraphs of sets which are uniform of degree~$d$ have about the expected count of all $(d+1)$-simple hypergraphs as subgraphs.
A simple proof of this fact will be given below, when we prove the Equivalence Theorem (Theorem~\ref{thm:Cayley_equiv}).

\smallskip

In order to state the main result of this section, we must first define a family of linear systems $(\Phi_{k, d})_{k\geq d+1}$ which are \emph{complete} for uniformity of degree~$d$:


\begin{definition}[The system $\Phi_{k, d}$]
    Let $\Phi_{k, d}$ be the linear system on $k 2^{\binom{k-1}{d}}$ variables and $2^{\binom{k}{d}}$ linear forms which is defined as follows.
    The variables are indexed by maps belonging to the set
    $$\Ical_{k,d} = \bigcup_{j=1}^k \bigg\{\tau: \binom{[k] \setminus \{j\}}{d} \rightarrow \{0, 1\} \bigg\}$$
    and, for each $\sigma \in \{0, 1\}^{\binom{[k]}{d}}$, there is a linear form $\phi_{\sigma}\in \Phi_{k, d}$ given by
    $$\phi_{\sigma}\big( (x_{\tau})_{\tau \in \Ical_{k,d}} \big) = \sum_{j=1}^k x_{\sigma_{-j}},$$
    where $\sigma_{-j}$ denotes the restriction of $\sigma$ to $\binom{[k] \setminus \{j\}}{d}$.
\end{definition}

In other words, we have variables indexed by partial maps from $\binom{[k]}{d}$ to $\{0, 1\}$ and forms indexed by (total) maps $\sigma: \binom{[k]}{d} \rightarrow \{0, 1\}$, with each form $\phi_{\sigma}$ being the sum of all variables indexed by those partial maps compatible with $\sigma$.
Note that the linear system $\Phi_{k, d}$ is the arithmetic analogue of the $d$-linear $k$-graph~$\M^{(k)}_d$ which is complete for quasirandomness of order~$d$:
a collection of elements $(x_1, \dots, x_K)\in G^K$ for $K = v(\M^{(k)}_d) = k 2^{\binom{k-1}{d}}$ forms a copy of~$\M^{(k)}_d$ in~$\Gamma^{(k)}_A$ iff $\phi(x_1, \dots, x_K)\in A$ for all $\phi\in \Phi_{k, d}$.
This can be easily shown from the construction of~$\M^{(k)}_d$ given in Section~\ref{sec:hyp_quasi}.

We are now ready to formally state and prove our Equivalence Theorem for Cayley hypergraphs, which was informally stated as Theorem~\ref{thm:equiv_informal} in the Introduction.


\begin{thm}[Equivalence Theorem for quasirandom Cayley hypergraphs] \label{thm:Cayley_equiv}
Let $k, d \geq 1$ be integers with $k\geq d+1$.
Then the following statements are polynomially equivalent, where $G$ is a finite additive group, $A \subseteq G$ is a subset and $\delta = |A|/|G|$:
\begin{itemize}
    \item[$(i)$] $A$ is uniform of degree $d$:\quad
    $\| A - \delta \|_{U^{d+1}} \leq c_1$.
    \item[$(ii)$] $A$ has few patterns of type $\Phi_{k, d}$:
    $$\mathbb{P}_{\xbf \in G^{\Ical_{k,d}}}\big( \Phi_{k, d}(\xbf) \subseteq A \big) \leq \delta^{2^{\binom{k}{d}}} + c_2.$$
    \item[$(iii)$] $\Gamma^{(k)}_A$ is quasirandom of order $d$:\quad
    $\disc_d(\Gamma^{(k)}_A) \leq c_3$.
    \item[$(iv)$] $\Gamma^{(k)}_A$ correctly counts all $d$-linear hypergraphs:
    $$t(F,\, \Gamma^{(k)}_A) = \delta^{e(F)} \pm e(F) c_4 \quad \text{for all } F \in \mathcal{L}^{(k)}_{d}.$$
    \item[$(v)$] $\Gamma^{(k)}_A$ correctly counts all $(d+1)$-simple hypergraphs:
    $$t(F,\, \Gamma^{(k)}_A) = \delta^{e(F)} \pm e(F) c_5 \quad \text{for all } F \in \mathcal{S}^{(k)}_{d+1}.$$
    \item[$(vi)$] $\Gamma^{(k)}_A$ has few squashed octahedra $\oct^{(k)}_{d+1}$:
    $$t\big( \oct^{(k)}_{d+1},\, \Gamma^{(k)}_A \big) \leq \delta^{2^{d+1}} + c_6.$$
    \item[$(vii)$] $\Gamma^{(k)}_A$ has small $(d+1)$-deviation:
    \quad $\dev_{d+1}\big(\Gamma^{(k)}_A\big) \leq c_7$.
\end{itemize}
\end{thm}

\begin{proof}
Several of the implications were already shown before.
Indeed, Theorem~\ref{thm:main} is a quantitative form of the equivalence $(i) \iff (iii)$, and Theorem~\ref{Towsnerthm} shows that $(ii)$, $(iii)$ and $(iv)$ are polynomially equivalent
(recall that $\Phi_{k,d}$ is precisely the additive counterpart to the hypergraph $\textsc{M}^{(k)}_d$).
We will next prove the implications
$$(i) \implies (v) \implies (vi) \implies (vii) \implies (i),$$
which will finish the proof of the theorem.

\medskip
$(i) \implies (v)$:
Suppose $\| A - \delta \|_{U^{d+1}} \leq c_1$, and let $F$ be a $(d+1)$-simple hypergraph.
Write $V = V(F)$, $E(F) = \{e_1, \dots, e_{e(F)}\}$, and for each $1 \leq i \leq e(F)$ let $f_i \subseteq e_i$ be a set of $d+1$ elements which is not contained in any other edge $e_j$.
By the usual telescoping sum argument we have
\begin{align*}
    &\big| t(F, \Gamma^{(k)}_A) - \delta^{e(F)} \big| \\
    &\quad\leq \sum_{i = 1}^{e(F)} \Bigg|\Exp_{\xbf_V \in G^{V}} \Bigg[ \big(A \circ \Sigma(\xbf_{e_i}) - \delta\big) \prod_{j = i+1}^{e(F)}{A \circ \Sigma(\xbf_{e_j})} \Bigg]\Bigg| \\
    &\quad\leq \sum_{i = 1}^{e(F)} \Exp_{\xbf_{V \setminus f_i}} \Bigg|\Exp_{\xbf_{f_i}} \Bigg[ \big(A \big( \Sigma(\xbf_{f_i}) + \Sigma(\xbf_{e_i \setminus f_i}) \big) - \delta\big)
    \prod_{j = i+1}^{e(F)} A \big( \Sigma(\xbf_{e_j \cap f_i}) + \Sigma(\xbf_{e_j \setminus f_i}) \big) \Bigg]\Bigg|.
\end{align*}

Consider the $i$-th term in the last sum.
For a fixed $\xbf_{V \setminus f_i} \in G^{V \setminus f_i}$ and each $i+1 \leq j \leq e(F)$, define on $G^{e_j \cap f_i}$ the function $u_j = u_{j, \xbf_{V \setminus f_i}} := \T^{\Sigma(\xbf_{e_j \setminus f_i})}A \circ \Sigma$;
by assumption $|e_j \cap f_i| \leq d$, while $|f_i| = d+1$.
The $i$-th term in the sum can then be written as
\begin{align*}
    \Exp_{\xbf_{V \setminus f_i}} \Bigg|\Exp_{\xbf_{f_i}} \Bigg[ \big(\T^{\Sigma(\xbf_{e_i \setminus f_i})}A \circ \Sigma(\xbf_{f_i}) - \delta\big)
    &\prod_{j = i+1}^{e(F)} u_j(\xbf_{e_j \cap f_i}) \Bigg]\Bigg| \\
    &\leq \Exp_{\xbf_{V \setminus f_i}} \big\| \T^{\Sigma(\xbf_{e_i \setminus f_i})}A \circ \Sigma - \delta \big\|_{\square^{d+1}_{d}} \\
    &= \| A \circ \Sigma - \delta \|_{\square^{d+1}_{d}},
\end{align*}
where we used translation invariance in the last equality.
It then follows that the whole sum is bounded by $e(F) \cdot \| A \circ \Sigma - \delta \|_{\square^{d+1}_{d}}$.
Item $(v)$ now follows from the fact that the octahedral norm is stronger than the cut norm (Lemma~\ref{lem:cut<oct}), since
$$\| A \circ \Sigma - \delta \|_{\square^{d+1}_{d}} = \big\| \Gamma^{(d+1)}_A - \delta \big\|_{\square^{d+1}_{d}} \leq \big\| \Gamma^{(d+1)}_A - \delta \big\|_{\oct^{d+1}} = \| A - \delta \|_{U^{d+1}} \leq c_1,$$
and we may take $c_5 = c_1$.

\medskip
$(v) \implies (vi)$:
This is a special case, and we may take $c_6 = 2^{d+1} c_5$.

\medskip
$(vi) \implies (vii)$:
First we note that $t(\oct^{(k)}_{d+1},\, \Gamma^{(k)}_f) = t(\oct^{(d+1)},\, \Gamma^{(d+1)}_f)$ holds for all functions $f: G \rightarrow \R$.
Indeed, we have that
\begin{align*}
    t\big(\oct^{(k)}_{d+1},\, \Gamma^{(k)}_f\big) &= \Exp_{\ybf \in G^{k-d-1}} \Exp_{\xbf^{(0)}, \xbf^{(1)} \in G^{d+1}} \prod_{\omega \in \{0, 1\}^{d+1}} f \big( \Sigma(\ybf) + \Sigma(\xbf^{(\omega)}) \big) \\
    &= \Exp_{\ybf \in G^{k-d-1}} t\big( \oct^{(d+1)},\, \Gamma^{(d+1)}_{\T^{\Sigma(\ybf)}f}\big) \\
    &= t\big(\oct^{(d+1)},\, \Gamma^{(d+1)}_f\big).
\end{align*}
Item $(vi)$ is then the same as requiring that $t\big(\oct^{(d+1)},\, \Gamma^{(d+1)}_A\big) \leq \delta^{2^{d+1}} + c_6$.
By the Equivalence Theorem for hypergraph quasirandomness and its proof, we conclude that $t\big(\oct^{(d+1)},\, \Gamma^{(d+1)}_A - \delta\big) \leq 2^{2^{d+2}} c_6^{1/2^{d+1}}$;
using the identity above for $f = A - \delta$, this is the same as saying that $t\big(\oct^{(k)}_{d+1},\, \Gamma^{(k)}_A - \delta\big) \leq 2^{2^{d+2}} c_6^{1/2^{d+1}}$, which is exactly item $(vii)$ with $c_7 = 2^{2^{d+2}} c_6^{1/2^{d+1}}$.

\medskip
$(vii) \implies (i)$:
As discussed in the previous equivalence, item $(vii)$ is the same as requiring that $t\big(\oct^{(d+1)},\, \Gamma^{(d+1)}_A - \delta\big) \leq c_3$.
Since
$$t\big(\oct^{(d+1)},\, \Gamma^{(d+1)}_A - \delta\big) = \big\| \Gamma^{(d+1)}_A - \delta \big\|_{\oct^{d+1}}^{2^{d+1}} = \|A - \delta\|_{U^{d+1}}^{2^{d+1}},$$
we obtain item $(i)$ with $c_1 = c_3^{1/2^{d+1}}$.
\end{proof}


\section*{Acknowledgements}

The work presented here was done while the author was at the University of Cologne.
It was supported by the European Union’s EU Framework Programme for Research and Innovation Horizon 2020 under the Marie Skłodowska-Curie Actions Grant Agreement No 764759 (MINOA), and by the Dutch Research Council (NWO) as part of the NETWORKS programme (grant no. 024.002.003).

\bibliographystyle{siam}
\bibliography{uniform}

\begin{thebibliography}{10}

\bibitem{QuasirandomnessHypergraphs}
{\sc E.~Aigner-Horev, D.~Conlon, H.~H\`an, Y.~Person, and M.~Schacht}, {\em
  Quasirandomness in hypergraphs}, Electron. J. Combin., 25 (2018), pp.~Paper
  No. 3.34, 22.

\bibitem{LinearQuasirandomness}
{\sc E.~Aigner-Horev and H.~H\`an}, {\em Linear quasi-randomness of subsets of
  abelian groups and hypergraphs}, European J. Combin., 88 (2020), pp.~103116,
  16.

\bibitem{QuasiSurvey}
{\sc D.~Castro-Silva}, {\em Quasirandomness in additive groups and hypergraphs:
  a survey}, 2021.

\bibitem{ChungSite}
{\sc F.~R.~K. Chung}, {\em References for quasirandom graphs in chronological
  order}.
\newblock \url{https://mathweb.ucsd.edu/~fan/qr/ref.html}.
\newblock Accessed on {March 16, 2023}.

\bibitem{QuasiClasses}
\leavevmode\vrule height 2pt depth -1.6pt width 23pt, {\em Quasi-random classes
  of hypergraphs}, Random Structures Algorithms, 1 (1990), pp.~363--382.

\bibitem{RegularityHypergraphsQuasirandomness}
\leavevmode\vrule height 2pt depth -1.6pt width 23pt, {\em Regularity lemmas
  for hypergraphs and quasi-randomness}, Random Structures Algorithms, 2
  (1991), pp.~241--252.

\bibitem{QuasirandomHypergraphs}
{\sc F.~R.~K. Chung and R.~L. Graham}, {\em Quasi-random hypergraphs}, Random
  Structures Algorithms, 1 (1990), pp.~105--124.

\bibitem{QuasiSetSystems}
\leavevmode\vrule height 2pt depth -1.6pt width 23pt, {\em Quasi-random set
  systems}, J. Amer. Math. Soc., 4 (1991), pp.~151--196.

\bibitem{QuasirandomZn}
\leavevmode\vrule height 2pt depth -1.6pt width 23pt, {\em Quasi-random subsets
  of {$\mathbb{Z}_n$}}, J. Combin. Theory Ser. A, 61 (1992), pp.~64--86.

\bibitem{QuasirandomGraphs}
{\sc F.~R.~K. Chung, R.~L. Graham, and R.~M. Wilson}, {\em Quasi-random
  graphs}, Combinatorica, 9 (1989), pp.~345--362.

\bibitem{WeakQuasirandomness}
{\sc D.~Conlon, H.~H\`an, Y.~Person, and M.~Schacht}, {\em Weak
  quasi-randomness for uniform hypergraphs}, Random Structures Algorithms, 40
  (2012), pp.~1--38.

\bibitem{FKM2013}
{\sc E.~Fouvry, E.~Kowalski, and P.~Michel}, {\em An inverse theorem for
  {G}owers norms of trace functions over {$\bold F_p$}}, Math. Proc. Cambridge
  Philos. Soc., 155 (2013), pp.~277--295.

\bibitem{NewProofLengthFour}
{\sc W.~T. Gowers}, {\em A new proof of {S}zemer\'{e}di's theorem for
  arithmetic progressions of length four}, Geom. Funct. Anal., 8 (1998),
  pp.~529--551.

\bibitem{NewProofSzemeredi}
\leavevmode\vrule height 2pt depth -1.6pt width 23pt, {\em A new proof of
  {S}zemer\'{e}di's theorem}, Geom. Funct. Anal., 11 (2001), pp.~465--588.

\bibitem{HypergraphRegularityGowers}
\leavevmode\vrule height 2pt depth -1.6pt width 23pt, {\em Hypergraph
  regularity and the multidimensional {S}zemer\'{e}di theorem}, Ann. of Math.
  (2), 166 (2007), pp.~897--946.

\bibitem{TrueComplexity}
{\sc W.~T. Gowers and J.~Wolf}, {\em The true complexity of a system of linear
  equations}, Proc. Lond. Math. Soc. (3), 100 (2010), pp.~155--176.

\bibitem{LinearEquationsPrimes}
{\sc B.~Green and T.~Tao}, {\em Linear equations in primes}, Ann. of Math. (2),
  171 (2010), pp.~1753--1850.

\bibitem{WeakHypergraphRegularity}
{\sc Y.~Kohayakawa, B.~Nagle, V.~R\"{o}dl, and M.~Schacht}, {\em Weak
  hypergraph regularity and linear hypergraphs}, J. Combin. Theory Ser. B, 100
  (2010), pp.~151--160.

\bibitem{HypergraphQuasirandomnessRegularity}
{\sc Y.~Kohayakawa, V.~R\"{o}dl, and J.~Skokan}, {\em Hypergraphs,
  quasi-randomness, and conditions for regularity}, J. Combin. Theory Ser. A,
  97 (2002), pp.~307--352.

\bibitem{PosetQuasirandomness}
{\sc J.~Lenz and D.~Mubayi}, {\em The poset of hypergraph quasirandomness},
  Random Structures Algorithms, 46 (2015), pp.~762--800.

\bibitem{SigmaAlgebrasHypergraphs}
{\sc H.~Towsner}, {\em {$\sigma$}-algebras for quasirandom hypergraphs}, Random
  Structures Algorithms, 50 (2017), pp.~114--139.

\end{thebibliography}

\end{document}